\documentclass[letterpaper,11pt]{amsart}

\usepackage{color,amscd}
\usepackage{enumerate}
\usepackage{mathtools}
\usepackage[abbrev]{amsrefs}
\usepackage{amssymb,amsmath,amsthm,xfrac,esint}

\newtheorem{thm}{Theorem}[section]
\newtheorem{prop}[thm]{Proposition}

\newtheorem{lem}[thm]{Lemma}

\newtheorem{remark}[thm]{Remark}
\newtheorem{corl}[thm]{Corollary}

\newcommand{\eps}{\varepsilon}

\def \<{\langle}
\def \>{\rangle}

\def \H{{\cal H}}

\def \H^0{{\cal H}^0 or}

\def \p{\partial}
\def \n{\nabla}
\def \beq{\begin{equation}}
\def \eeq{\end{equation}}

\def \n{\nabla}

\begin{document}



\title[Cheeger's isoperimetric constant]{A note on Cheeger's isoperimetric constant}

\author{Nelia Charalambous}
\address{Department of Mathematics and Statistics, University of Cyprus, Nicosia, 1678, Cyprus} \email[Nelia Charalambous]{nelia@ucy.ac.cy}

\author{Zhiqin Lu} \address{Department of
Mathematics, University of California,
Irvine, Irvine, CA 92697, USA} \email[Zhiqin Lu]{zlu@uci.edu}

\thanks{The first author was partially supported by a University of Cyprus Internal grant. The second author is partially supported by the DMS-19-08513.}
 \date{\today}




\begin{abstract}
In this short exposition we provide a simplified proof of Buser's result for Cheeger's isoperimetric constant. We also provide a comprehensive approach on how to obtain volume estimates for smooth hypersurfaces.
\end{abstract}
\maketitle

{\centering\footnotesize \it This article is dedicated to Peter Li on the occasion of his 70th birthday.\par}

\section{Introduction}

Let $(M^n,g)$ be a  compact   Riemannian manifold without boundary. Then Cheeger's constant is defined as
\[
h(M)=\inf \frac{\text{Vol}_{n-1}(\Sigma)}{\min\{\text{Vol}_n(A),\text{Vol}_n(B)\} },
\]
where  $\Sigma$ is a hypersurface which divides $M$ into two disjoint open sets $A, B$ such that $\bar{A}\cup\bar{B}=M$ and $\p A=\p B=\Sigma$.
It is well-known that
\[
\lambda_1(M)\geq \frac 14 h(M)^2,
\]
where $\lambda_1(M)$ is the first nonzero eigenvalue of the Laplacian on functions over $M$. Conversely, in \cite{Bus} Buser proved the following result.

\begin{thm}[Buser \cite{Bus}] \label{thm1} Suppose that $M$ is a smooth compact manifold with Ricci curvature bounded below, $\text{Ric} \geq -(n-1)K$  for some nonnegative constant $K$. Then
\[
\lambda_1(M)\leq C(n) \, \left( \sqrt{K} \;h(M) + h^2(M)\right),
\]
where   $C(n)$ is a constant that depends only on the dimension of the manifold.\end{thm}

Our main goal in this note is to provide a simplified proof of this theorem.  Buser shows this result by proving a lower bound for the isoperimetric constant of Dirichlet regions (see \cite{Bus}*{Lemma 5.1}). Instead, in Lemma \ref{lem4} we will prove a lower bound for
\[
\frac{{\rm Vol}_{n-1} (\Sigma\cap B_x(3r))}{\min({\rm Vol}_n (A\cap B_x(r)),{\rm Vol}_n (B\cap B_x(r)))}
\]
whenever $\Sigma$ is a smooth hypersurface that splits $M$ into the regions $A, B$; $x\in M$; and  $r>0$ is appropriately chosen.  This estimate does not imply a lower bound for the isoperimetric constant of a geodesic ball $B_x(r)$, since in the numerator we intersect the hypersurface with a ball that is three times larger. As we will see, this more \emph{flexible} comparison about the area of the hypersurface in a larger geodesic ball allows us to avoid having to resort to the consideration of Dirichlet regions which made Buser's  proof of Theorem 1.1 more complicated in Section 4 of his paper.  For an alternative approach on how to prove Buser's upper estimate   heat kernel estimates see \cite{led}.

We also use this occasion to provide a comprehensive approach on how to obtain volume estimates for smooth hypersurfaces. In Proposition \ref{prop1} we prove a lower bound for $ \frac{{\rm Vol}_{n-1}(\p M)}{{\rm Vol}_n(M)}$ whenever $M$ is a manifold with boundary,  assuming only an integrability condition  on the Ricci curvature of the manifold and the mean curvature of the boundary (or equivalently the hypersurface). We note that there is no assumption on the compactness of either $M$ or the boundary.

We begin with some preliminary facts about the distance function to a hypersurface.  Suppose that $\Sigma$ is a smooth oriented hypersurface in  $M$.  Let $\rho(x)$ denote the signed distance of the point $x$ from $\Sigma$ such that
\[
|\rho(x)|= \text{dist}(x,\Sigma) := \inf\{ d(x,y) \mid  y\in \Sigma \,\}.
\]
It is well known that $\rho$ is continuous  (by appropriately choosing its sign on either side of $\Sigma$) and $|\n \rho|=1$ except at the focal points of $\Sigma$, the points where the normal exponential map fails to be an immersion \cite{Esch}.

\begin{lem}
Suppose that $\Sigma$ is a smooth oriented hypersurface in a smooth manifold $M$.  Define
\begin{equation*}
\begin{split}
\mathcal{S} := & \{ x\in M \mid  \exists! y\in  \Sigma  \ \text{such that} \ d(x,y)= \mathrm{dist}(x, \Sigma), \\
& \qquad \qquad \,\text{and the geodesic connecting } x,y \text{ is unique}\}.
\end{split}
\end{equation*}
Then,
\begin{enumerate}
\item  $M\setminus \mathcal{S}$ has zero measure, and
\item $\mathcal{S}$ is starlike, in the sense that if $x \in \mathcal{S}$ and $y=y(x)$ is the unique point in $\Sigma$ such that $d(x,y(x))= \mathrm{dist}(x, \Sigma)$, then the geodesic $\overline{xy} \subset \mathcal{S}$.
\end{enumerate}
\end{lem}

By definition, $\text{Hess}\,\rho$ is the covariant derivative of $\n \rho=\p \rho$, the normal direction.
 At the same time on a level set $\rho^{-1} (r)$,  $\text{Hess} \,\rho= I\!I$ corresponds to the second fundamental form of the level set, and $\Delta \rho = m$ is its mean curvature. Note that the choice of sign for $\rho$ also affects the sign of the mean curvature of the level set.

In Section \ref{S2} we provide comparison results for the signed distance function $\rho$ that will lead to volume estimates depending on the mean curvature of the hypersurface. In Section \ref{S3} we will provide the simplified proof of Theorem \ref{thm1}. In this case the mean curvature of the hypersurface does not appear in the volume comparison estimates.

\section{Comparison Results} \label{S2}

In this section we will review some comparison results for the signed distance function $\rho$, which will lead to volume estimates depending on the mean curvature of the hypersurface. We use this result to prove a lower bound for ${\text{Vol}_{n-1}(\p M)}/{\text{Vol}_n(M)}$ on a manifold with boundary. By the Bochner formula,
\begin{equation}
\begin{split}
\tfrac 12 \Delta |\n \rho|^2 &= |\text{Hess}\, \rho|^2 + \langle \n \rho, \n (\Delta \rho) \rangle + \text{Ric} (\n \rho,\n \rho) \Rightarrow\\
0& = |I\!I|^2 + \frac{\p m}{\p \rho} +\text{Ric} (\p \rho,\p \rho).
\end{split}
\end{equation}

The Hessian of the function $\rho$ has an eigenvalue which is zero, because  $(\text{Hess} \, \rho)\p \rho=0$. By the Cauchy-Schwarz inequality we have $|\text{Hess} \, \rho|^2 \geq \tfrac{1}{n-1} \, (\Delta \rho)^2 = \frac{m^2}{n-1}$ and we denote $\frac{\p m}{\p \rho}=m'$. Then
\[
m' \leq -\frac{m^2}{n-1} - \text{Ric} (\p \rho,\p \rho).
\]

Under the assumption of $\text{Ric} \geq -(n-1)K$ for some $K\geq 0$, we have
\begin{equation}\label{ricc2}
m' \leq -\frac{m^2}{n-1}+(n-1)K.
\end{equation}
When $K=0$, this equation implies that $m$ is nonincreasing. For points where $m \neq 0$ we have
\[
\left( - \frac{1}{m} \right)' \leq - \frac{1}{n-1} +(n-1)K.
\]

For constants $H, K$, we consider the function $\psi_{K,H}$ which solves the Riccati equation:
\begin{equation}\label{ricc}
\psi_{K,H}' +\frac{\psi_{K,H}^2}{n-1}-(n-1)K=0,\quad \psi_{K,H}(0)=H.
\end{equation}

We recall the following result for the Riccati equation
\begin{lem}
If $K=0$, then
 \[
 \psi_{0,H}(t)=\frac{(n-1)H}{(n-1)+tH}.
 \]
In general, if $K > 0$, then the solution is given by
\[
\psi_{K,H}(t)=(n-1)\sqrt K\cdot\frac{(n-1)\sqrt K\sinh\sqrt Kt+H\cosh\sqrt Kt}{(n-1)\sqrt K\cosh\sqrt Kt+H\sinh\sqrt Kt}.
\]
In particular, whenever $(n-1)\sqrt K-H=0$, then $\psi(t)\equiv H$.

Let $T\in (0,\infty]$ correspond to the maximal time interval of existence  for $\psi_{K,H}(t)$.

If $H>0$, then $T=\infty$; if $H<0$, then
\[
T=\frac{1}{\sqrt K}\tanh^{-1}\frac{(n-1)\sqrt K}{-H} \ \ \text{for} \ \ K>0 \ \ \text{and} \ \ T=\frac{(n-1)}{-H}  \ \ \text{for} \ \ K=0.
\]
\end{lem}

\bigskip

\begin{lem} \label{lem2}
Suppose that  $M$ is a complete manifold with  $\text{Ric} \geq -(n-1)K$ for some $K\geq 0$, and let $\Sigma$ be a smooth oriented hypersurface in $M$. Let $y(x)\in \Sigma$  such that $|\rho(x)|=d(x,y)$ and assume that $\Sigma$ has mean curvature $H(y(x))$ at $y$. If $\rho$ is differentiable at $x$, then for all $x$ with $\rho(x)\geq 0$
\[
\Delta \rho(x) \leq \psi_{K,H} (\rho(x))
\]
where $\psi_{K,H}$ is the solution to \eqref{ricc}.  The above comparison holds within the maximal interval of existence for $\psi_{K,H}$.

Moreover, for all $x$ with $\rho(x)\leq 0$,
\[
\Delta \rho(x) \geq \psi_{K,H} (\rho(x))
\]
provided that whenever $H>0$,   $\rho(x)$ satisfies
\[
\rho(x) \geq \frac{1}{\sqrt K}\tanh^{-1}\frac{(n-1)\sqrt K}{-H}.
\]

In addition,  the same inequalities for $\Delta \rho$ hold in the sense of distribution.
\end{lem}

\begin{proof}
In the case $\rho(x)\geq 0$ consider a flow line of the mean curvature $m(\rho)$ along $\rho$, with $m(0)=H$.  By the Bochner inequality,  $m = \Delta \rho(x)$ satisfies inequality \eqref{ricc2}. Then, by the Riccati equation comparison theorem $m(r) \leq \psi_{K,H}(r)$ in $\mathcal{S}$ \cite{Esch}*{Theorem 4.1}.  Eschenburg also proves that the interval of existence for $m$ contains that of $\psi$ and the singularities of the equation are only vertical asymptotes.

For the case $\rho(x)\leq 0,$ we define $\rho_1(x)=-\rho(x)$ and observe
\[
\Delta \rho_1(x) \leq \psi_{K,-H} (\rho_1(x)) =-\psi_{K,H} (\rho(x)).
\]
Since $\Delta \rho_1(x)=-\Delta \rho(x)$, the inequality follows immediately, and the interval for which the inequality holds is the corresponding maximal interval of existence of $\psi_{K,H} (\rho(x))$.

To prove the inequality in the sense of distribution we consider the case $\rho(x)\geq 0$ and let $\mathcal{S}_o\subset \mathbb{R}^{n-1}$ be a maximal starlike domain such that the exponential map
\[
\exp_{\Sigma} (\mathcal{S}_o) \to M
\]
is a diffeomorphism onto its image.

Let $\Omega_\eps \subset  \exp_{\Sigma}(\mathcal{S}_o)$ be a starshaped domain such that $\Omega_\eps \to  \exp_{\Sigma}(\mathcal{S}_o)$ as $\eps \to 0$. Let $\phi\geq$ in $C_o^\infty$. By the above inequality,
\[
\int_{\Omega_\eps} \phi \Delta \rho \leq \int_{\Omega_\eps} \phi \, \psi.
\]

Integration by parts gives
\[
\int_{\Omega_\eps} \phi \Delta \rho = - \int_{\Omega_\eps} \n \phi \cdot \n \rho +\int_{\p \Omega_\eps} \phi \frac{\p \rho}{\p \eta}
\]

In the classical argument, since $\Omega_\eps$ is star-shaped and $\phi \geq 0$ the last term in nonnegative. The same should be true here. Therefore, after sending $\eps \to 0$,
\[
\int_{\exp_{\Sigma}( \mathcal{S}_o)} \phi \Delta \rho \geq  - \int_{\exp_{\Sigma}( \mathcal{S}_o)} \n \phi \cdot \n \rho.
\]

Since $\rho$ is Lipschitz, then it is differentiable almost everywhere, and its derivative coincides with its weak derivative in the $H^1$ sense. By the definition of the $H^1$ weak derivative we would then get
\[
 - \int_{M} \n \phi \cdot \n \rho =\int_{M} \phi \, \Delta \rho.
\]

The case for $\rho\leq 0$ is done similarly, by considering the positive function $\rho_1=-\rho$.
\end{proof}

We have the following upper bounds for the solution to the Riccati equation.

\begin{corl} \label{lem3}
Let  $\psi_{K,H}(t)$ denote the solution to the Riccati equation \eqref{ricc}. For $\rho(x)\geq 0$, whenever the initial condition is nonpositive, $H\leq 0$, then
\[
\Delta \rho(x) \leq (n-1) \sqrt{K},
\]
and whenever the initial condition is positive, $H > 0$, then
\[
\Delta \rho(x) \leq \max\{H, (n-1) \sqrt{K}\}.
\]

For $\rho(x)\leq 0$, if the initial condition is nonnegative, $H\geq 0$, then
\[
-\Delta \rho(x) \leq (n-1) \sqrt{K}
\]
and whenever the initial condition is negative, $H < 0$, then
\[
-\Delta \rho(x) \leq \max\{-H, (n-1) \sqrt{K}\}.
\]
\end{corl}
\begin{proof}
Note that $\psi(t)=\psi_{K,H}(t)$ for $t\geq 0$  can only have singularities of the type $-\infty$ \cite{Esch}*{Section 2}. Moreover, since $\psi(t)$ satisfies the Riccati equation, we know that at its extremal points $\psi'(t_o)=0$ hence $\psi^2(t_o) = (n-1)^2K$. Also note that as $t\to +\infty$ $\psi_{K,H}\to (n-1) \sqrt{K}$. As a result, whenever $H\leq 0$, $\psi_{K,H}(t)\leq (n-1) \sqrt{K}$ for all $t\geq0$ on the interval where the solution exists.  $\psi_{K,H}(t)$ can decrease to $-\infty$, but if it becomes positive, it cannot go above the nonnegative bound $(n-1) \sqrt{K}$.

On the other hand, whenever $H>0$ and very large, then $\psi_{K,H}(t)$  must decrease to $(n-1) \sqrt{K}$  for $t>0$. In this case, for negative $\rho$, taking $\rho_1 = -\rho$ we get that $\Delta \rho_1$ has initial value $-H<0$ and by the proof of Lemma \ref{lem2}
\[
\Delta \rho_1(x) \leq \psi_{K,-H} (\rho_1(x))= (n-1)\sqrt K\cdot\frac{(n-1)\sqrt K\sinh\sqrt K\rho_1 -H\cosh\sqrt K\rho_1}{(n-1)\sqrt K\cosh\sqrt K\rho_1 -H\sinh\sqrt K\rho_1}.
\]
In this case, the right side decreases to $-\infty$ in finite time, while remaining bounded above by $(n-1) \sqrt{K}$ for all $\rho_1>0$.
\end{proof}

The above upper bounds illustrate that the Laplacian of $\rho$ (for $\rho$ positive) before the focal points remains bounded, and when it becomes singular at the focal points it tends to negative infinity. This behavior is similar to the Laplacian of the distance function to a point before and at the cut locus.

We are now ready to discuss some volume estimate results related to the signed distance function $\rho$.  Consider a set $W$ where $\rho\geq 0$ and $\Delta \rho \leq C$ is bounded. Then
\[
\int_W |\Delta \rho| \leq \int_W  \left(|C- \Delta \rho| + C \right)   = \int_W (2C- \Delta \rho) = 2 C\,\text{Vol}_n(W) + \text{Vol}_{n-1}(\partial W).
\]

Denote $A_{a,b} :=\{ x \mid 0\leq a\leq \rho(x) \leq b \,\}$, and $S_a:=\{ x\mid \rho(x) =a \,\}$. Define $V(a,b)=\text{Vol}_{n} (A_{a,b})$ and  $f(a) = \text{Vol}_{n-1} (S_a)$.

Suppose that $\Delta \rho \leq C$. Then
\[
f(t) - f(0) \leq C\int_{0}^t f(s) \, ds.
\]
Hence
\[
\left( e^{-Ct} \int_{0}^t f(s) \, ds \right)' \leq e^{-Ct} f(0)
\]
and
\[
 e^{-Ct} V(0,t) \leq \frac{1}{C} f(0)   \Rightarrow \frac{f(0)}{V(0,t)} \geq C e^{-Ct}.
\]

By slightly generalizing the above argument, we can now prove the following.

\begin{prop} \label{prop1}
Suppose that  $M$ is a complete manifold with smooth boundary $\p M$ such that $\text{Ric} \geq -(n-1)K$ for some $K\geq 0$, and $\text{diam}(M) \leq D$. Assume that $\p M$ the following quantity is bounded
\[
C_o=\int_{\p M} \max\left\{H(\theta), (n-1) \sqrt{K}\right\} \, d\theta.
\]
Then,
\[
\frac{{\rm Vol}_{n-1}(\p M)}{{\rm Vol}_n(M)}\geq C_o e^{-C_o \,D}
\]
Note that $M$ need not be a compact manifold.
\end{prop}

 Observe that we get the same lower bound for $\ {\rm Vol}_{n-1}(\Sigma)/\min\{\text{Vol}_n(A),\text{Vol}_n(B)\}\ $ if $\Sigma$ is a smooth oriented hypersurface which splits $M$ into two sets $A, B$.

\begin{proof}
Let $y(x)\in \p M$  such that $\rho(x)=d(x,y)$ and assume that $\p M$ has mean curvature $H(y(x))$ at $y$. If $\rho(x)$ is differentiable at $x$, then
\[
\Delta \rho (x) \leq \psi_{K,H} (\rho (x))
\]
along the flow line from $x$ to $y(x)$.

By Corollary \ref{lem3},  $\psi_{K,H} \rho (x)\leq \max\left\{H(y(x)), (n-1) \sqrt{K}\right\}.$  As a result
\[
f(t) - f(0) \leq    \int_{0}^t \int_{\p M} \max\left\{H(y(x)), (n-1) \sqrt{K}\right\}\, d\theta \, ds \leq  C_o \, V(0,t)
\]
and
\[
\frac{f(0)}{V(0,t)}\geq C_o e^{-C_o \,D}
\]
where $D$ is the diameter of $M$.

We remark that the above estimate holds on the regions where the exponential map from $\p M$ is a diffeomorphism, but we can extend the comparison result by using a similar argument as in \cite{zhu}. In the case of Ricci nonnegative, the assumption of Proposition \ref{prop1} reduces to $\int_{\p M} \max\left\{H(y(x)),  0\right\} \leq C_o$.

\end{proof}

\section{A simplified  proof of  Buser's result} \label{S3}

In this section we provide the proof of Buser's estimate without resorting to his  consideration of Dirichlet regions. Let $\Sigma$ be a  smooth hypersurface in $M$ which divides it into two disjoint  open regions $A, B$. Let
\[
\mathfrak h = \frac{\text{Vol}_{n-1}(\Sigma)}{\min\{\text{Vol}_n(A),\text{Vol}_n(B)\} }.
\]
For simplicity we will omit the subscript for the dimension of the volume for the rest of the paper.  In addition, we shall use $C(n)$, $\eps(n)$ to denote a general constant that depends only on $n$, but which need not be the same throughout.

The key lemma of this paper that can be used to replace  Lemma 5.1 in \cite{Bus} is the following.
\begin{lem} \label{lem4}
There is a constant $C=C(n)$ such that for any
  $r>0$ and $x\in M$, we have
\[
{\rm Vol} (\Sigma\cap B_x(3r))\geq \frac{C(n)}{r}e^{-3(n-1)\sqrt{K} r}\min({\rm Vol} (A\cap B_x(r)),{\rm Vol} (B\cap B_x(r))).
\]
\end{lem}

\begin{proof}
Let $\hat A=A\cap B_x(r)$ and $\hat B=B\cap B_x(r)$.  Note that both $\hat A$, $\hat B$ may not be connected in general.
For a pair of points $p \in \hat A, \; q\in \hat B$ let $w\in \Sigma$ be the first point on the minimizing geodesic from $p$ to $q$  when the geodesic intersects $\Sigma$.
It is clear that $w\in \Sigma\cap B_x(3r)$,
and the geodesic has to be completely contained in $B_x(3r)$.
Similarly, let $\tilde{w}$ be the first point on the same minimizing geodesic in the opposite direction (from $q$ to $p$)  when the geodesic intersects $\Sigma\cap B_x(3r)$. Since it is well-known that the set of pairs $(p,q)\in \hat A\times\hat B$ which do not have a unique geodesic has measure zero, it follows that for almost all pairs $(p,q)$ both $w$ and $\tilde w$ are well-defined.

Define the  following subsets of $\hat A\times \hat B$
\begin{align*}
&
W_0 =\{ (p,q) \in \hat A\times \hat B \mid d(q, w) \geq d(p,w) \};\\
&
W_1 =\{ (p,q) \in \hat A\times \hat B \mid d(p, \tilde{w}) \geq d(q,\tilde{w}) \}.
\end{align*}
Then $W_0 \cup W_1$  covers $\hat A\times \hat B$ up to a set of measure zero. Define the projections
\[
\pi_o: \hat A\times \hat B\to \hat A,\qquad \pi_1: \hat A\times \hat B\to \hat B.
\]
Then, without loss of generality we may assume that
\[
\text{Vol}(W_1) \geq  \frac 12 \, \text{Vol}(\hat A) \, \text{Vol}(\hat B).
\]

Let $V_0=\pi_0 (W_1)$ be the projection of the set $W_1$ on $\hat A$. Then from
\[
\text{Vol}(W_1) = \int_{V_0} \int_{\pi_1(\pi_0^{-1}(x) \cap W_1)} dv(y) \, dv(x) \geq \frac 12 \, \text{Vol}(\hat A) \, \text{Vol}(\hat B),
\]
 it follows that there exists at least one point $z\in \hat  A$ such that $F=\pi_1(\pi_0^{-1}(z) \cap W_1)\subset \hat{B}$ satisfies
\begin{equation}\label{45}
\text{Vol}(F)\geq \int_F  dv \geq \frac 12   \, \text{Vol}(\hat B)\geq \frac 12\min\{\text{Vol}(\hat A), \text{Vol}(\hat B)\}.
\end{equation}

Fix any $q\in F$ and let $\overline{qz}$ be the minimizing geodesic from $q$ to $z$. Let $\tilde w$ be  the first point where the geodesic $\overline{qz}$  intersects  $\Sigma$ in the direction from $q$ to $z$. Since $(z,q)\in W_1$, $d(z, {\tilde w}) \geq d(q, {\tilde w})$.
Let $\sigma=d(z,\tilde w)$. Let $\mathcal{R}$ be an infinitesimal  radial annular sector defined by $\overline{zq}$ with center $z$.
Then by  volume comparison (see for example \cite{zhu}) we have
\[
\text{Vol}(\Sigma\cap \mathcal{R})   \geq \frac{\sinh^{n-1}(\sqrt{K} \sigma)}{\int_\sigma^{2\sigma} \sinh^{n-1}(\sqrt{K} t)\, dt} \; \text{Vol}(F\cap \mathcal R) \geq  \frac{C(n)}{\sigma} \; e^{-(n-1)\sqrt{K} \sigma} \; \text{Vol}(F\cap \mathcal R).
\]
Since $\sigma\leq 3r$, we have
\[
\text{Vol}(\Sigma\cap \mathcal{R})   \geq\frac{C(n)}{r} \; e^{-3(n-1)\sqrt{K} r} \; \text{Vol}(F\cap \mathcal R).
\]
Adding over all the sectors $\mathcal{R}$ that intersect $\Sigma\cap  B_x(3r)$ (and hence $F$), we have
\[
\text{Vol}(\Sigma\cap B_x(3r))    \geq\frac{C(n)}{r} \; e^{-3(n-1)\sqrt{K} r} \; \text{Vol}(F),
\]
where $C(n)$ is a constant that depends only on $n$.  The lemma then follows from the above inequality and ~\eqref{45}.
\end{proof}

Note that if we take $r=D$, the diameter of manifold $M$, we would get a lower bound of Cheeger's constant
\[
h(M)\geq\frac{C(n)}{D}e^{-3(n-1)D}.
\]
In the following, we strengthen Lemma~\ref{lem4} to prove Theorem~\ref{thm1}, Buser's result.

\begin{thm}
There exists a $C(n)>0$ such that
\[
\lambda_1(M)\leq C(n)\,\frac{\mathfrak h}{\min\{K^{-1/2},\mathfrak h^{-1}\}},
\]
where $\lambda_1(M)$ is the first eigenvalue of the compact manifold $M$.
\end{thm}

\begin{proof}
We let
\[
r=\eps(n)\min\{K^{-1/2},\mathfrak h^{-1}\},
\]
where $\eps(n)$ is a small constant that will be  determined  later. Using Lemma~\ref{lem4} for such a choice of $r$ we have
\[
{\rm Vol} (\Sigma\cap B_x(3r))\geq \frac{C(n)}{r} \min\{{\rm Vol} (A\cap B_x(r)),{\rm Vol} (B\cap B_x(r))\}.
\]

Following Buser, we define the sets
\[
\tilde \Sigma=\{x\in M \mid  {\rm Vol}(A\cap B_x(r))={\rm Vol}(B\cap B_x(r)) \},
\]
and
\begin{equation*}
\begin{split}
& \tilde A=\{x\in M \mid  {\rm Vol}(A\cap B_x(r))> \frac 12 \,{\rm Vol}(B_x(r)) \},\\
     & \tilde B=\{x\in M\mid  {\rm Vol}(B\cap B_x(r))> \frac 12 \,{\rm Vol}(B_x(r)) \}.
\end{split}
\end{equation*}

We define a cover of $M$ by $r$-balls $B_{p_i}(r)$ for $i=1,\ldots, k$ such that the centers $p_i$ are at least $r$ away from each other. Since $r\leq \eps(n) K^{-1/2}$,  we can assume that the balls $B_{p_i}(3r)$  overlap at most $C_o(n)$ times at each point of $M$ (in other words our cover is a Gromov cover). With respect to $\tilde\Sigma, \tilde A, \tilde B$, we
can choose the Gromov  cover  such that  $p_1,\ldots ,p_s \in \tilde \Sigma $  and $\tilde \Sigma$ is covered by the balls $B_{p_i}(r)$ for $i=1,\ldots , s$; for   $i=s+1,\ldots , m$,  $p_i \in \tilde B$; and for $i=m+1,\ldots ,k$,  $p_i \in \tilde A$.

We claim that neither $\tilde A$ nor $\tilde B$ is empty. If, for example, $\tilde B=\emptyset$, then  ${\rm Vol} (A\cap B_x(r))\geq {\rm Vol} (B\cap B_x(r))$  for all $x\in M$ and hence
\[
{\rm Vol} (\Sigma\cap B_{p_i}(3r))\geq \frac{C(n)}{r}  {\rm Vol} (B\cap B_{p_i}(r))
\]
for any $i$ by Lemma~\ref{lem4}.
Summing over $p_i$   we shall get
\[
{\rm Vol} (\Sigma )\geq C(n)  \sum_{i=1}^{k} {\rm Vol} (\Sigma\cap B_{p_i}(3r) )\geq\sum_{ i=1}^{k}\frac{C(n)}{r}{\rm Vol} (B \cap B_{p_i}(r) )
\geq \frac{C(n)}{r}{\rm Vol} (B ).
\]
But by definition,
\[
{\rm Vol} (B )\geq \min\{{\rm Vol} (A ),{\rm Vol} (B )\}=\mathfrak h^{-1} {\rm Vol} (\Sigma ).
\]
So we get
\[
1\geq \frac{C(n)\mathfrak h^{-1}}{r},
\]
which is a contradiction if we choose $\eps(n)$ sufficiently small,  and the claim is proved.

Note that $\tilde A$ and $\tilde B$ are separated by $\tilde \Sigma$. Therefore if  neither of $\tilde A, \tilde B$ is empty, then $\tilde\Sigma\neq\emptyset$.\\

Let $t>0$. Define
\[
{\tilde \Sigma}^t=\{x\in M\mid  {\rm dist} (x, \tilde \Sigma) \leq t \}.
\]
 Then ${\tilde \Sigma}^t$ is covered by the balls $B_{p_i}(r+t)$ for $i=1,\ldots , s$. By volume comparison,  ${\rm Vol}(B_{p_i}(4r))/{\rm Vol}(B_{p_i}(r)) \leq C(n)$, hence using  the definition of $\tilde\Sigma$, we have
 \[
{\rm Vol}({\tilde \Sigma}^{3r}) \leq C(n)\sum_{i=1}^s {\rm Vol}(B_{p_i}(r)) \leq  C(n) \sum_{i=1}^s {\rm Vol}(A\cap B_{p_i}(r)),
\]
 since ${\rm Vol}(A\cap B_{p_i}(r))= {\rm Vol}(B\cap B_{p_i}(r))$ for $i=1,\ldots , s$. By Lemma~\ref{lem4} we can bound each term in the right side to get
\[
{\rm Vol}({\tilde \Sigma}^{3r}) \leq C(n)\sum_{i=1}^s {\rm Vol}(A\cap B_{p_i}(r))\leq
C(n)\,r\sum_{i=1}^s{\rm Vol}(\Sigma\cap B_{p_i}(3r)).
\]
 By our assumption on the Gromov cover, the balls $B_{p_i}(3r)$ overlap at most $C_o(n)$ times at each point, therefore
\[
\sum_{i=1}^s{\rm Vol}(\Sigma\cap B_{p_i}(3r)) \leq C_o(n) {\rm Vol}(\Sigma).
\]
By the definition of $\mathfrak h$ we get
\begin{equation}\label{e2}
{\rm Vol}({\tilde \Sigma}^{3r})  \leq C(n)\mathfrak h \, r \, \min\{{\rm Vol}(A), {\rm Vol}(B)  \}.
\end{equation}

If $q\in \tilde B\setminus \tilde\Sigma^{r}$, then $B_q(r) \subset \tilde B$, and $q$ is in the part of the cover for $s+1\leq i \leq m$. Therefore,
\begin{equation} \label{e3}
\begin{split}
{\rm Vol}(\tilde A\setminus {\tilde \Sigma}^{r}) & \geq{\rm Vol}(A \cap (\tilde A\setminus {\tilde\Sigma}^{r})\,) = {\rm Vol}(A) - {\rm Vol}(A\cap (\tilde B \cup {\tilde \Sigma}^{r})) \\
& \geq {\rm Vol}(A) - {\rm Vol}({\tilde \Sigma}^{r}) -  \sum_{i=s+1}^m {\rm Vol}(A\cap  B_{p_i}(r)).  \\
\end{split}
\end{equation}

Since $p_i \in \tilde B$, ${\rm Vol}(A\cap  B_{p_i}(r))\leq {\rm Vol}(B\cap  B_{p_i}(r))$. Hence by Lemma \ref{lem4} and using the maximum overlap of the Gromov cover we get
\begin{equation*}
\begin{split}
\sum_{i=s+1}^m {\rm Vol}(A\cap  B_{p_i}(r)) & \leq C(n) \, r \sum_{i=s+1}^m {\rm Vol}(\Sigma \cap  B_{p_i}(3r)) \leq C(n) \, r {\rm Vol}(\Sigma) \\
&= C(n)\mathfrak h \, r \, \min\{{\rm Vol}(A), {\rm Vol}(B)  \}.
\end{split}
\end{equation*}

Finally, substituting the above estimate and \eqref{e2} into the right side of \eqref{e3} we get
\[
{\rm Vol}(\tilde A\setminus {\tilde \Sigma}^{r}) \geq  {\rm Vol}(A) - C(n)\mathfrak h \, r \, \min\{{\rm Vol}(A), {\rm Vol}(B)  \}.
\]
Similarly, we have
\[
{\rm Vol}(\tilde B\setminus {\tilde \Sigma}^{r}) \geq  {\rm Vol}(B) - C(n)\mathfrak h \, r \, \min\{{\rm Vol}(A), {\rm Vol}(B)  \}.
\]
 By taking complements we also get
\[
{\rm Vol}(\tilde A\setminus {\tilde \Sigma}^{r})-{\rm Vol}(A) \leq -({\rm Vol}(\tilde B\setminus {\tilde \Sigma}^{r}) -{\rm Vol}(B))\leq C(n)\mathfrak h \, r \, \min\{{\rm Vol}(A), {\rm Vol}(B)  \},
\]
and as a result,
\begin{align*}
&
|{\rm Vol}(\tilde A\setminus {\tilde \Sigma}^{r})-{\rm Vol}(A)|\leq C(n) \, r \,  {\mathfrak h} \, \min\{{\rm Vol}(A), {\rm Vol}(B)  \};\\
&
|{\rm Vol}(\tilde B\setminus {\tilde \Sigma}^{r})-{\rm Vol}(B)|\leq C(n) \, r \,  {\mathfrak h} \, \min\{{\rm Vol}(A), {\rm Vol}(B)  \}.
\end{align*}

We  let $\rho(x)$ denote the distance function to  $\tilde \Sigma$  ($\rho$ is taken to be nonnegative here),  and define the function
\[
f(x)=\left\{
\begin{array}{ll}
 {\rm Vol}(\tilde B) & {\rm if } \  x\in \tilde A \setminus {\tilde \Sigma}^{r};
\\[1ex]
\frac{\displaystyle \rho(x)}{\displaystyle  r} {\rm Vol}(\tilde B) & {\rm if } \  x\in \tilde A \cap {\tilde \Sigma}^{r};\\[1ex]
-\frac{\displaystyle \rho(x)}{\displaystyle  r} {\rm Vol}(\tilde A) & {\rm if }  \  x\in \tilde B \cap{\tilde \Sigma}^{r};\\[1ex]
-{\rm Vol}(\tilde A) & {\rm if } \  x\in \tilde B \setminus{\tilde \Sigma}^{r}.
\end{array}
\right.
\]
 Note that even if $\tilde \Sigma$ is not a smooth hypersurface, the function $\rho$ is still well-defined, a.e. continuous and its gradient also exists a.e.  Let $\tilde f(x)$ be defined such that $\tilde f(x)= {\rm Vol}(\tilde B)$ if $x\in\tilde A$ and
$\tilde f(x)= -{\rm Vol}(\tilde A)$ if $x\in\tilde B$. Then
\[
|f(x)-\tilde f(x)|\leq  {\rm Vol}(M)
\]
on $\tilde \Sigma^r$ and is $0$ otherwise. It is now easy to see that
\[
\left|\int_M  f(x)\right|=\left|\int_M  (f(x)-\tilde f(x))\right|\leq  {\rm Vol}(\tilde\Sigma^r)\cdot {\rm Vol}(M).
\]

Moreover,
\[
\int_M f(x)^2\geq \left(1 -  C(n) \, r \,  {\mathfrak h} \right){\rm Vol}(A)\,{\rm Vol}(B)\,{\rm Vol}(M).
\]

Thus
\begin{align}\label{e4}
\begin{split}
\int_M \left(f - \fint_M f\right)^2 &=\int_M f^2- {\rm Vol}(M)^{-1}\left(\int_M  f(x)\right)^2\\
&\geq\left(1 -  C'(n) ( r  {\mathfrak h}+( r {\mathfrak h})^2)  \, \right){\rm Vol}(A)\,{\rm Vol}(B)\,{\rm Vol}(M),
\end{split}
\end{align}
where $\fint f =({\rm Vol}(M))^{-1} \int_M f$. We also have
\[
\int_M |\nabla f(x)|^2\leq \frac{C(n)}{r^2} {\rm Vol}(M)^2 \;{\rm Vol}( \tilde{\Sigma}^r ).
\]
It is clear from \eqref{e2} that
\[
\frac{ {\rm Vol}(M) {\rm Vol}(\tilde{\Sigma}^r)}{{\rm Vol}(A)\,{\rm Vol}(B) }\leq
C(n) \, r \,  {\mathfrak h}.
\]

Finally, we choose $\varepsilon(n)$ even smaller if necessary in the definition of $r$ so that $\varepsilon(n) C'(n) \leq \frac 14$, where $C'(n)$ is the constant in \eqref{e4}. Then $ C'(n) \,( r \, {\mathfrak h} + r^2 \, {\mathfrak h}^2)\leq \frac 12$.  By the variational principle, we have
\[
\lambda_1(M)\leq \frac{\int_M  |\nabla f|^2}{\int_M (f - \fint_M f)^2} \leq C(n) \,  {\mathfrak h}/r.
\]
The theorem is proved.
\end{proof}

\begin{proof}[Proof of Theorem \ref{thm1}]
From the definition of $r$ we immediately get
\[
\lambda_1(M)\leq C(n) (\sqrt K\, {\mathfrak h}+ {\mathfrak h}^2).
\]
In particular,  since we can choose $\mathfrak h$ arbitrarily close to $h(M)$, Theorem~\ref{thm1} follows.
\end{proof}

\begin{remark}
Define the logarithmic isoperimetric constant,
\[
k(M)=\inf \frac{{\rm Vol}_{n-1}(\Sigma)}{{\rm Vol}_n(A)|\log ({\rm Vol}_n(A))|},
\]
where ${\rm Vol}_n(A)\leq {\rm Vol}_n(B)$. Let $\rho_0$ be the optimal constant in the log-Sobolev inequality. Then by our method, we can prove a similar result to Theorem~\ref{thm1}:
\[
\rho_0\leq C(n)(\sqrt Kk(M)+k(M)^2).
\]
The above result was first proved by Ledoux~\cite{led}*{Theorem 2}.
\end{remark}

After this paper was written, we were informed by Professor G. Wei that a similar result had been proved in the papers of Dai-Wei-Zhang. In particular, Corollary 4.3 in~\cite{d1} is very similar to our Lemma~\ref{lem4} (their result is slightly more general in that the assumption for a lower bound on Ricci curvature  is replaced by an almost lower bound in the integral sense). In ~\cite{d2}, they have considered the similar problem over manifolds with convex boundary.

We thank Professor G. Wei for bringing these papers to our attention.



\begin{bibdiv}
\begin{biblist}

\bib{Bus}{article}{
   author={Buser, Peter},
   title={A note on the isoperimetric constant},
   journal={Ann. Sci. \'{E}cole Norm. Sup. (4)},
   volume={15},
   date={1982},
   number={2},
   pages={213--230},
   issn={0012-9593},
   review={\MR{683635}},
}
	
	
	\bib{d1}{article}{
   author={Dai, Xianzhe},
   author={Wei, Guofang},
   author={Zhang, Zhenlei},
   title={Local Sobolev constant estimate for integral Ricci curvature
   bounds},
   journal={Adv. Math.},
   volume={325},
   date={2018},
   pages={1--33},

}


\bib{d2}{article}{
   author={Dai, Xianzhe},
   author={Wei, Guofang},
   author={Zhang, Zhenlei},
   title={Neumann isoperimetric constant estimate for convex domains},
   journal={Proc. Amer. Math. Soc.},
   volume={146},
   date={2018},
   number={8},
   pages={3509--3514},

}
	


\bib{Esch}{article}{
   author={Eschenburg, J.-H.},
   title={Comparison theorems and hypersurfaces},
   journal={Manuscripta Math.},
   volume={59},
   date={1987},
   number={3},
   pages={295--323},
   issn={0025-2611},
   review={\MR{909847}},
   doi={10.1007/BF01174796},
}

\bib{led}{article}{
   author={Ledoux, M.},
   title={A simple analytic proof of an inequality by P. Buser},
   journal={Proc. Amer. Math. Soc.},
   volume={121},
   date={1994},
   number={3},
   pages={951--959},
   issn={0002-9939},
   review={\MR{1186991}},
   doi={10.2307/2160298},
}
	
\bib{zhu}{article}{
   author={Zhu, Shunhui},
   title={The comparison geometry of Ricci curvature},
   conference={
      title={Comparison geometry},
      address={Berkeley, CA},
      date={1993--94},
   },
   book={
      series={Math. Sci. Res. Inst. Publ.},
      volume={30},
      publisher={Cambridge Univ. Press, Cambridge},
   },
   date={1997},
   pages={221--262},
   review={\MR{1452876}},
}
	


\end{biblist}
\end{bibdiv}
\end{document}